\documentclass[12pt,a4paper]{article}
\usepackage{amssymb, amscd, amsthm, amsmath, latexsym, amstext, mathrsfs}
\usepackage[all]{xy}
\usepackage{graphicx}
\usepackage[colorlinks, linkcolor=blue, anchorcolor=green, citecolor=blue]{hyperref}

\newtheorem{thm}{Theorem}[section]
\newtheorem{coro}{Corollary}[section]
\newtheorem{lemma}{Lemma}[section]

\theoremstyle{remark}
\newtheorem{remark}{\textbf{Remark}}[section]

\theoremstyle{definition}

\newtheorem{example}{Example}[section]

\numberwithin{equation}{section}
\newcommand{\cO}{\mathcal{O}}

\newcommand{\N}{\mathbb{N}}
\newcommand{\Z}{\mathbb{Z}}
\newcommand{\Q}{\mathbb{Q}}

\newcommand{\ord}{\mathrm{ord}}

\newcommand{\rep}{\rightharpoonup}

\begin{document}

\title{\textbf{Pythagoras number of quartic orders containing $\sqrt{2}$}}
\author{Zilong HE and Yong HU}
\date{}

\maketitle %\thispagestyle{empty}

\begin{abstract}
Let $K$ be a quartic number field containing $\sqrt{2}$ and let $\mathcal{O}\subseteq K$ be an order such that $\sqrt{2}\in \mathcal{O}$. We prove that the Pythagoras number of $\mathcal{O}$ is at most 5. This confirms a conjecture of Kr\'{a}sensk\'{y}, Ra\v{s}ka and Sgallov\'a. The proof makes use of Beli's theory of bases of norm generators for quadratic lattices over dyadic local fields.
\end{abstract}

%\tableofcontents

%\newpage

%\section{Statements of main results}

\section{Introduction}

Let $R$ be a commutative ring and let $\Sigma R^2$ denote the set of all elements that can be written as a sum of finitely many squares of elements of $R$. The \emph{Pythagoras number} $\mathcal{P}(R)$ of $R$ is defined as the smallest positive integer $m$ or $\infty$ such that every element in $\Sigma R^2$ can be written as a sum of at most $m$ squares of elements in $R$. The Pythagoras number of fields has been extensively studied  in the literature. The case when $R$ is an order in a number field $K$ is another case of interest to many researchers. In this case, obtaining an upper bound for $\mathcal{P}(R)$ is generally much more difficult than doing the same for $\mathcal{P}(K)$. The interested readers are referred to \cite[Chapter\;7]{Pfister95}, \cite{KrasenskyRaskaSgallova22}, \cite{Tinkova23}, \cite{Krasensky22ArchMath} and the references therein for more information on this topic.

In a recent paper \cite{KrasenskyRaskaSgallova22}, Kr\'asensk\'y, Ra\v{s}ka and Sgallov\'a studied the Pythagoras number of the maximal order $\cO_K$ in a totally real biquadratic number field $K$. They proved, among others, that $\mathcal{P}(\cO_K)\ge 5$ with at most seven exceptions and that  $\mathcal{P}(\cO_K)\le 5$ if $\sqrt{5}\in \cO_K$. Based on results of numerical calculations with the help of computers and by a theoretical analysis of analogy with the case $\sqrt{5}\in \cO_K$, they conjecture that $\mathcal{P}(\cO_K)\le 5$ when $\sqrt{2}\in \cO_K$ (\cite[Conjecture\;1.6 (2)]{KrasenskyRaskaSgallova22}). The goal of this note is to confirm this conjecture. More precisely, we prove the following  analog of \cite[Theorem\;1.3]{KrasenskyRaskaSgallova22}:

\begin{thm}\label{1.1}
  Let $K$ be a quadratic extension of $\Q(\sqrt{2})$ and let $\mathcal{O}\subseteq K$ be an order such that $\sqrt{2}\in \mathcal{O}$.   Then $\mathcal{P}(\cO)\le 5$.
\end{thm}

A strategy for proving Theorem\;\ref{1.1} has already been suggested in \cite[\S\;8.3]{KrasenskyRaskaSgallova22}. The suggestion is mainly motivated by the fact that the genus of the quadratic form $I_4=\langle 1,1,1,1\rangle$ over $\Z[\sqrt{2}]$ consists of a single class, just as over the maximal order of $\Q(\sqrt{5})$. By a generalization of \cite[Cor.\;3.3]{KalaYatsyna21AdvMath} (see \cite[Prop.\;7.5]{KrasenskyRaskaSgallova22}), a key step turns out to be the determination of the invariant $g_R(2)$ for $R=\Z[\sqrt{2}]$. (We will recall the definition of this invariant in \S\;\ref{sec4}.) For the maximal order $R$ of the field $\Q(\sqrt{5})$, the invariant $g_R(2)$ is known from Sasaki's work \cite{Sasaki05}. When working with $\Q(\sqrt{2})$ instead of $\Q(\sqrt{5})$, the main difficulty results from the ramification of the prime number 2, which makes it less easy to analyze integral quadratic forms at the dyadic place.

To determine $g_R(2)$ for $R=\Z[\sqrt{2}]$, we need a characterization of binary forms that are representable as a sum of 4 linear forms over $R$. In \S\S\;\ref{sec2} and \ref{sec3}, we derive this missing ingredient from the representation theory of integral quadratic forms over general dyadic fields. This is made possible by the theory of bases of norm generators developed by Beli (see e.g. \cite{Beli01thesis} and \cite{Beli19}). We feel that this method, though important and powerful, has not been widely used in the literature.

As expected, the result is $g_R(2)=5$ for $R=\Z[\sqrt{2}]$ (Theorem\;\ref{4.1}). From this Theorem\;\ref{1.1} follows by \cite[Prop.\;7.5]{KrasenskyRaskaSgallova22}.

\

\noindent {\bf Terminology and notation.} In the rest of the paper, we use the geometric language of quadratic spaces and lattices in the study of quadratic forms. We refer the reader to O'Meara's book \cite{OMeara00} for standard terminology and notation about them. In particular, the scale and the norm of a quadratic lattice $M$ are denoted by $\mathfrak{s}(M)$ and $\mathfrak{n}(M)$ respectively.

Unless otherwise stated, all quadratic spaces and lattices are assumed to be nonsingular.

Given nonzero elements $a_1,\cdots, a_n$ in a field $F$, we denote by $[a_1,\cdots, a_n]$ the quadratic space defined by the diagonal quadratic form $a_1x_1^2+\cdots+a_nx_n^2$ on the vector space $F^{\oplus n}:=F\oplus\cdots\oplus F$. Similarly, when a Dedekind domain $R$ is fixed,  let $\langle a_1,\cdots, a_n\rangle$ denote the quadratic lattice defined by the same quadratic form on the module $R^{\oplus n}$. For any integer $m\ge 1$, let $I_m$ denote the rank $m$ lattice $\langle 1,\cdots, 1\rangle $.

If $M,\,N$ are quadratic lattices over $R$, we write $N\rep M$ if $N$ is represented by $M$. Similarly for representations of quadratic spaces.

\section{A representability criterion over dyadic fields}\label{sec2}

The main result in this section, Theorem\;\ref{2.5}, is a criterion for the representability of a binary lattice as a sum of linear forms over a general dyadic local field. This result will be deduced from a general representation theorem in the theory of bases of norm generators, as developed by Beli in a series of papers \cite{Beli01thesis,Beli03JNT,Beli06,Beli10TAMS,Beli19}.

Let us briefly review  some key definitions and facts  that will be used in this paper.  The reader is referred to Beli's papers for any	unexplained notation and definition.

Throughout this section, let $F$ be an arbitrary dyadic local field, i.e. a finite extension of the field $\Q_2$ of 2-adic numbers.
Let $\ord: F\to \Z\cup\{\infty\}$ denote the normalized discrete valuation on $F$. We write $\cO_F$ for the valuation ring of $F$ and put $e=\ord(2)$. For any $c\in F^\times:=F\setminus\{0\}$, let $\mathfrak{d}(c)=\bigcap_{x\in F}(c-x^2)\cO_F$. The function
\[
d\;:\; F^\times\longrightarrow \N\cup\{\infty\}\;;\;c\longmapsto d(c):=\min\{\ord(c^{-1}z)\,|\,z\in \mathfrak{d}(c)\}
\]is the called the \emph{order of relative quadratic defect} (\cite[p.127, Definition\;1]{Beli03JNT}). It is well known that
\[
d(F^\times)=\{0,\,2e,\,\infty\}\cup \{1,\,3,\,5,\cdots, 2e-1\}\;.
\]

\begin{lemma}\label{2.1}
  With notation as above, we have $d(-1)\ge e$.
\end{lemma}
\begin{proof}
  From the definition, we have $\mathfrak{d}(-1)=\mathfrak{d}(1+(-2))\subseteq 2\cO_F$. Hence $d(-1)\ge \ord(2)=e$.
\end{proof}

\medskip

Let $ M$ be a quadratic $ \mathcal{O}_{F} $-lattice with associated quadratic form $Q$. A vector $ x\in M$ is called a \textit{norm generator} of $ M $ if $ \mathfrak{n}M=Q(x)\mathcal{O}_{F} $. A sequence of vectors $ x_{1},\ldots,x_{m}$ in $FM$ is called a \textit{Basis Of Norm Generators} (BONG) for $M$ if $ x_{1} $ is a norm generator for $ M $ and $ x_{2},\ldots,x_{m} $ is a BONG for $ \mathrm{pr}_{x_{1}^{\perp}}M$, where $ \mathrm{pr}_{x_{1}^{\perp}} $ denotes the projection from $ FM$ to $(Fx_{1})^{\perp} $, the orthogonal complement of $Fx_1$ in $FM$.
		
		A BONG $ x_{1},\ldots,x_{m} $ is said to be \textit{good}, if $ \ord\, Q(x_{i})\le \ord\,Q(x_{i+2})$ for all $ 1\le i\le m-2 $.
By \cite[Corollary 2.6]{Beli03JNT}, a lattice is uniquely determined by a BONG.  Also, every lattice possesses a good BONG (see \cite[Lemma 4.6]{Beli03JNT} for a proof and \cite[\S 7]{Beli06} for an algorithm).

		We will write $M\cong \prec a_{1},\ldots,a_{m}\succ $ to mean that $M$ has a good BONG $ x_{1},\ldots,x_{m} $ such that $a_i=Q(x_i)$. Using such a good BONG,
we define the $R$-\emph{invariants}
 \[
 R_{i}(M):=\ord(a_i)=\ord(Q(x_i))\,,\;1\le i\le m
  \] and the $\alpha$-\emph{invariants}
  \[
  \begin{split}
  \alpha_i(M):=\min\bigg(\bigg\{\frac{R_{i+1}-R_i}{2}+e\bigg\}&\cup \big\{R_{i+1}-R_j+d(-a_ja_{j+1})\,|\;1\le j\le i\big\}\\
  &\cup \big\{R_{j+1}-R_i+d(-a_ja_{j+1})\,|\,i\le j<m\big\}\bigg)
  \end{split}
  \]for $1\le i\le m-1$. These invariants  are independent of the choice of the good BONG (\cite[Thm. 3.1]{Beli10TAMS}).

\begin{example}\label{2.3}
  Let $m$ be a positive integer and $M=I_m$. It is easy to see that $M\cong \prec a_1,\cdots, a_m\succ$ with
  $a_1=\cdots=a_m=1$. Note that $e\le d(-1)$ by Lemma\;\ref{2.1}. Hence the $R$-invariants and the $\alpha$-invariants of $M$ are given by
  \[
  R_1=R_2=\cdots=R_m=0\quad\text{and}\quad \alpha_1=\cdots=\alpha_{m-1}=\min\{e,\,d(-1)\}=e\,.
  \]

  Let $N$ be a binary $\cO_F$-lattice and assume $N\cong \prec b_1,\,b_2\succ$. We put $S_i=R_i(N)$ and $\beta_1=\alpha_1(N)$. Then $S_i=\ord(b_i)$ and
\[
\beta_1=\min\left\{\frac{S_2-S_1}{2}+e\,,\,S_2-S_1+d(-b_1b_2)\right\}\,.
\]
\end{example}	

To a pair of $\cO_F$-lattices $M$ and $N$ with $n=\mathrm{rank}(N)\le m=\mathrm{rank}(M)$, one can associate the $A$-\emph{invariants} $A_i(M,\,N)$ for $1\le i\le \min\{m-1,\,n\}$ as in \cite[Definition\;4.3]{Beli06}. We do not repeat the definition in the general case, but only give explicit formulas in a special case needed in this paper.

Let $N\cong \prec b_1,\,b_2\succ$ and $M=I_m$ with $m\ge 4$. With the same notation as in Example\;\ref{2.3}, assume further that $S_1\ge 0$ and $S_1+S_2\ge 0$.

Then the two invariants $A_1=A_1(M,N)$ and $A_2=A_2(M,N)$ are given by
\[
A_1=\min\bigg\{-\frac{S_1}{2}+e\,,\;-S_1+e\,,\;-S_1+d(-1)\bigg\}=e-S_1
\]and if $m=4$,
\[
\begin{split}
A_2&=\min\bigg\{-\frac{S_2}{2}+e\,,\;-S_2+\min\{d(-b_1),\,\alpha_3,\,\beta_1\}\bigg\}\\
&=\min\bigg\{e-\frac{S_2}{2}\,,\;\min\{d(-b_1),\,e\}-S_2\,,\;e-\frac{S_1+S_2}{2}\,,\;d(-b_1b_2)-S_1\bigg\}\\
&=\min\bigg\{d(-b_1)-S_2,\,e-S_2,\;e-\frac{S_1+S_2}{2}\,,\;d(-b_1b_2)-S_1\bigg\}
\end{split}
\]
or if $m\ge 5$,
\[
\begin{split}
A_2&=\min\bigg\{-\frac{S_2}{2}+e\,,\;-S_2+\min\{d(-b_1),\,\alpha_3,\,\beta_1\}\,,\;-S_1-S_2+\alpha_4\bigg\}\\
&=\min\bigg\{d(-b_1)-S_2,\,e-S_2,\;e-\frac{S_1+S_2}{2}\,,\;d(-b_1b_2)-S_1\,,\,\alpha_4-S_1-S_2\bigg\}\\
&=\min\bigg\{d(-b_1)-S_2,\,d(-b_1b_2)-S_1\,,\;e-S_1-S_2\bigg\}\,.
\end{split}
\]

We need the following special case of \cite[Theorem 4.5]{Beli06}.

	\begin{thm}\label{2.5}
Let $N\cong \prec b_1,\,b_2\succ$ and $M=I_m$ with $m\ge 4$ as above.

Then $ N\rep M $ if and only if  $ FN\rep FM $ and the following conditions hold:

\begin{enumerate}
  \item $S_1\ge 0$ and $S_1+S_2\ge 0$.
  \item $ d(-b_1)\ge e-S_1$ and $d(-b_1b_2)\ge e-S_2$.
  \item If $m=4$, $S_2<0$, $d(-b_1)+d(-b_1b_2)>2e+S_2$ and $2d(-b_1b_2)>2e+S_1$, then the binary space $[b_1,\,b_2]$ is represented by the ternary space $[1,\,1,\,1]$.
\end{enumerate}
	\end{thm}
\begin{proof}
  First, it is easy to see that in our situation condition (i) in \cite[Theorem 4.5]{Beli06} is the same as our condition (1). Condition (ii) of \cite[Theorem 4.5]{Beli06}  reads
  \begin{equation}\label{2.5.1}
      \min\{d(b_1),\,\alpha_1,\,\beta_1\}\ge A_1\quad \text{and}\quad \min\{d(b_1b_2),\,\alpha_2\}\ge A_2\;.
  \end{equation}  As we have seen above, we have $\alpha_1=\alpha_2=e$ and $A_1=e-S_1$. Assuming (1), we have $\alpha_1\ge A_1$ and
  \[
  A_2=\min\bigg\{d(-b_1)-S_2,\,d(-b_1b_2)-S_1\,,\;e-S_1-S_2\bigg\}\le e-S_1-S_2\le \alpha_2=e\,.
  \]Thus, the first inequality in \eqref{2.5.1} is equivalent to the two inequalities $d(b_1)\ge e-S_1$ and $\beta_1=\min\{\frac{S_2-S_1}{2}+e,\,S_2-S_1+d(-b_1b_2)\}\ge e-S_1$. Since $S_1+S_2\ge 0$, we have $\beta_1\ge e-S_1$ if and only if $d(-b_1b_2)\ge e-S_2$.

  Since $d(-1)\ge e\ge \max\{A_2,\,e-S_1\}$, by the domination principle for the function $d$ (\cite[Lemma\;1.1]{Beli03JNT}), the inequalities  $d(b_1)\ge e-S_1$ and $d(b_1b_2)\ge A_2$ are equivalent to $d(-b_1)\ge e-S_1$ and  $d(-b_1b_2)\ge A_2$ respectively.  So \eqref{2.5.1} holds if and only if
  \[
  d(-b_1)\ge e-S_1\,,\; d(-b_1b_2)\ge e-S_2 \quad\text{ and } \quad d(-b_1b_2)\ge A_2\,.
  \]In fact, the first two inequalities here imply the third, because they imply that
  \[
  A_2=\min\bigg\{d(-b_1)-S_2,\,d(-b_1b_2)-S_1\,,\;e-S_1-S_2\bigg\}=e-S_1-S_2\le e-S_2\,.
  \]Thus, we see that condition (ii) of \cite[Theorem 4.5]{Beli06} can be translated into condition (2) in our theorem (when assuming (1)).

 Note that the inequality $0>S_1$ does not hold by (1). By \cite[p.6, Remarks (1)]{Beli20},  condition (iii) of \cite[Theorem 4.5]{Beli06} can be rephrased as follows: \emph{If $S_2<0$ and $d[-a_{13}b_{11}]+d[-a_{14}b_{12}]>2e+S_2$, then $[b_1,\,b_2]\rep [1,1,1]$, where in our context}
 \[
d[-a_{13}b_{11}]=
\min\bigg\{d(-b_1),\,e,\,e+\frac{S_2-S_1}{2},\,S_2-S_1+d(-b_1b_2)\bigg\}
\]and
\[
d[-a_{14}b_{12}]=\begin{cases}
  d(-b_1b_2)\quad & \text{ if }  m=4\,,\\
  \min\{d(-b_1b_2),\,e\}\quad & \text{ if } m\ge 5\,.
\end{cases}
\]
In the case $S_2<0$, from (1) and (2) we get $e+\frac{S_2-S_1}{2}<e$ and $d(-b_1b_2)\ge e-S_2>e$, whence
\[
d[-a_{13}b_{11}]=
\min\bigg\{d(-b_1),\,e+\frac{S_2-S_1}{2},\,S_2-S_1+d(-b_1b_2)\bigg\}
\]
and
\[
d[-a_{14}b_{12}]=\begin{cases}
  d(-b_1b_2)\quad & \text{ if }  m=4\,,\\
 e\quad & \text{ if } m\ge 5\,.
\end{cases}
\]Note that $S_1+S_2\ge 0$ by (1). Thus, if $S_2<0$ and $m\ge 5$, we have
\[
d[-a_{13}b_{11}]+d[-a_{14}b_{12}]\le e+\frac{S_2-S_1}{2}+e=2e+S_2-\frac{S_1+S_2}{2}\le 2e+S_2\;.
\]So there is no need to check condition (iii) of \cite[Theorem 4.5]{Beli06} if $m\ge 5$.

If $S_2<0$ and $m=4$, the inequality $d[-a_{13}b_{11}]+d[-a_{14}b_{12}]>2e+S_2$ means
\[
\min\bigg\{d(-b_1),\,e+\frac{S_2-S_1}{2},\,S_2-S_1+d(-b_1b_2)\bigg\}+d(-b_1b_2)>2e+S_2\,.
\]If $S_2-S_1+d(-b_1b_2)+d(-b_1b_2)>2e+S_2$, we have $d(-b_1b_2)>e+\frac{S_1}{2}$ and thus
\[
e+\frac{S_2-S_1}{2}+d(-b_1b_2)>e+\frac{S_2-S_1}{2}+e+\frac{S_1}{2}=2e+\frac{S_2}{2}>2e+S_2\,.
\]From this we see that condition (iii) of \cite[Theorem 4.5]{Beli06} is equivalent to our condition (3).

Since the inequality $R_{i+2}>R_{i+1}+2e$ does not hold in our context, there is no need to check condition (iv) of \cite[Theorem 4.5]{Beli06}. The theorem is thus proved.
\end{proof}

\begin{coro}\label{2.6}
Let $N\cong \prec b_1,\,b_2\succ$ and $S_i=\ord(b_i)$.

Then the following assertions are equivalent:

\begin{enumerate}
  \item $N\rep I_5$.
  \item $N\rep I_m$ for some $m\ge 2$.
  \item $N\rep I_m$ for some $m\ge 5$.
  \item The following two conditions hold:
  \begin{enumerate}
  \item $S_1\ge 0$ and $S_1+S_2\ge 0$.
  \item $ d(-b_1)\ge e-S_1$ and $d(-b_1b_2)\ge e-S_2$.
\end{enumerate}
\end{enumerate}
\end{coro}
\begin{proof}
  Trivially, (1)$\Rightarrow$(2). Since $I_m\rep I_{m+3}$, we have (2)$\Rightarrow$(3). A quadratic space of dimension at least 5 represents all binary spaces (\cite[63:21]{OMeara00}). So the condition $FN\rep FM$ holds automatically if $M=I_m$ with $m\ge 5$. Thus, from Theorem\;\ref{2.5} we see that (3)$\Rightarrow$(4)$\Rightarrow$(1).
\end{proof}

\begin{remark}
  To test representability by $I_m$ over any dyadic local field, one can also use the Third Main Theorem in Riehm's work \cite{Riehm64AJM}, which relies only on the classical invariants of lattices as presented in \cite{OMeara00}. However, we feel that our criterion in Theorem\;\ref{2.5} (especially in the case $m=4$), obtained by using the theory of BONGs, is more convenient for our purpose.
\end{remark}

\section{Sums of 4 squares of linear forms over $\mathbb{Z}_2[\sqrt{2}]$}\label{sec3}

In this section, let $F=\Q_2(\sqrt{2})$.

\medskip

Let us recall some useful facts about the field $F$.

Clearly, a uniformizer in $F$ is $\sqrt{2}$ and the ramification index $e=\ord(2)$ is 2. Note that $-1=(1+\sqrt{2})^2-(4+2\sqrt{2})$ and $\ord(4+2\sqrt{2})=3<4=2e$. So we have $d(-1)=3$ by \cite[63:5]{OMeara00}.

Since $[F:\Q_2]$ is even, $-1$ is a sum of two squares in $F$ (\cite[Chapter 3, 1.2(6)]{Pfister95}). Thus, the ternary space $[1,1,1]$ is isotropic and the quaternary space $[1,1,1,1]$ is hyperbolic over $F$. In particular, $[1,1,1,1]$ represents all binary spaces over $F$.

The unique quadratic unramified extension of $F$ is $F(\sqrt{5})$. So by \cite[63:3 and 63:4]{OMeara00}, an element $c\in F^{\times}$ satisfies $d(c)=4=2e$ if and only if $c\in 5F^{\times 2}$, i.e. $c/5$  is a square in $F$.

\begin{lemma}\label{3.2}
  Let $Q$ be a binary quadratic form with coefficients in $\cO_F$. Suppose that $Q$ is a sum of finitely many squares of linear forms over $\cO_F$.

    Then $Q$ is not a sum of $4$ squares of linear forms over $\cO_F$ if and only if $Q$ is equivalent to the form $G(x,y):=2\sqrt{2}(x^2+xy+y^2)$.
\end{lemma}
\begin{proof}
  Let $N$ be the lattice defined by $Q$ on the module $\cO_F\oplus\cO_F$. Assume $N\cong \prec b_1,\,b_2\succ$ and put $S_i=\ord(b_i)$.  As we have mentioned at the beginning of this section, $FN$ is represented by  $[1,1,1,1]$. Here we have assumed that $N$ is represented by $I_m$ for some (large) $m$. So by Theorem\;\ref{2.5} (and Corollary\;\ref{2.6}), $N$ is not represented by $I_4$ if and only if the following conditions hold:

   (1) $S_1\ge -S_2>0$.

   (2) $d(-b_1)\ge 2-S_1$ and $d(-b_1b_2)\ge 2-S_2$.

   (3) $d(-b_1)+d(-b_1b_2)>4+S_2$.

   (4) $2d(-b_1b_2)>4+S_1$.

   (5) The binary space $[b_1,\,b_2]$ is not represented by $[1,1,1]$.

We claim that the above 5 conditions are equivalent to
\[
S_1=3\,,\;S_2=-1\quad\text{and}\quad d(-b_1b_2)=4\,.
\]
We will use the fact that $d(F^{\times})=\{0,\,4,\,\infty,\,1,\,3\}$. As we have said at the beginning of this section, the equality $d(-b_1b_2)=4$ means that $-b_2\in 5b_1 F^{\times 2}$. When $S_1=3$, this implies that the Hilbert symbol $(-b_1,\,-b_2)_F$ is equal to $(-b_1,\,5)_F=-1$ (\cite[63:11a]{OMeara00}). But  the binary space $[b_1,\,b_2]$ is represented by $[1,1,1]\cong [b_1,\,-b_1,\,-1]$ if and only if the Hilbert symbol $(-b_1,\,-b_2)_F$ equals 1. From this the sufficiency part of the claim follows easily.

Now let us prove the necessity part of the claim.

If $S_1=1$, then we have $d(-b_1)=0$, contradicting the first inequality in (2). So we must have $S_1\ge 2$. This combined with (4) yields $d(-b_1b_2)>3$. On the other hand, since $[1,1,1]$ is isotropic, condition (5) implies that $-b_1b_2$ is not a square in $F$, or equivalently $d(-b_1b_2)<\infty$. Hence $d(-b_1b_2)=4$.

From the second inequality in (2) we deduce that $S_2\ge -2$, and from (4) we see $S_1<4$. We have already shown $S_1\ge 2$. Thus by (1) we have $S_2\in \{-1,\,-2\}$ and $S_1\in \{2,\,3\}$. As we mentioned in the proof of sufficiency, the condition $d(-b_1b_2)=4$ implies that $(-b_1,\,-b_2)_F=(-b_1,\,5)_F$. If $S_1=2$, the Hilbert symbol $(-b_1,\,5)_F$ equals 1, which leads to a contradiction to (5). Therefore, $S_1=3$.

If $S_2=-2$, then $S_1+S_2$ is odd, which implies $d(-b_1b_2)=0$, a contradiction. So we have $S_2=-1$. This proves our claim.

Now, using \cite[Cor.\;3.4 (iii)]{Beli03JNT} we can conclude that $N$ has scale $\mathfrak{s}(N)=\sqrt{2}\cO_F$ and norm $\mathfrak{n}(N)=2\sqrt{2}\cO_F$. Note that the space $FN$ is anisotropic. By \cite[93:11]{OMeara00}, $N$ is isomorphic to the lattice represented by the matrix $\sqrt{2}A(2,2)=\begin{pmatrix}
  2\sqrt{2} & \sqrt{2}\\
  \sqrt{2} & 2\sqrt{2}
\end{pmatrix}$. This is exactly the binary lattice corresponding to the binary form $2\sqrt{2}(x^2+xy+y^2)$. The lemma is thus proved.
\end{proof}

\begin{remark}
In Lemma\;\ref{3.2} one cannot drop the assumption that $Q$ is a sum of finitely many squares of linear forms over $\cO_F$. In fact, the diagonal form $\sqrt{2}x^2+\sqrt{2}y^2$ is not a sum of squares of linear forms over $\cO_F=\Z_2[\sqrt{2}]$.

In general, if $E$ is a finite extension of $\Q_2$ in which $2$ is ramified, then the uniformizer $\pi$ of $E$ is not a sum of squares in $\cO_E$, and hence the diagonal lattice $\langle \pi,\,\pi\rangle$ is not represented by $I_m$ for any $m\ge 1$. To see this, we can consider  the norm group $\mathfrak{g}(I_m)$. We have
$\mathfrak{g}(I_m)=\mathfrak{g}(\langle 1\rangle)=\cO_E^2+2\cO_E$.
If $\mathrm{ord}_E$ denotes the normalized discrete valuation on $E$, then $\mathrm{ord}_E(\pi-\alpha^2)\in\{0,1\}$ for all $\alpha\in \cO_E$. So we have $\pi\notin \mathfrak{g}(I_m)=\cO_E^2+2\cO_E$ when $\ord_E(2)>1$.
\end{remark}

\section{Proof of main result}\label{sec4}

For any commutative ring $R$ and any positive integer $k$, the invariant $g_R(k)$ is  defined as the smallest positive integer $n$ (or $\infty$) such that every $k$-ary quadratic form that can be written as a finite sum of squares of linear forms over $R$ is a sum of $n$ squares of linear forms.

As we have said in the introduction, Theorem\;\ref{1.1} follows from the following result, which is an analog of a theorem of Sasaki \cite{Sasaki05} (see also \cite[Thm.\;7.7]{KrasenskyRaskaSgallova22}).

\begin{thm}\label{4.1}
  We have $g_{\Z[\sqrt{2}]}(2)=5$.
\end{thm}
\begin{proof}
  The binary form $f(x,y):=(16+2\sqrt{2})(x^2+xy+y^2)$ is a sum of linear forms over $\Z[\sqrt{2}]$, in view of the identity
  \[
  f=((1+\sqrt{2})x+y)^2+(1+\sqrt{2})^2y^2+(x+y)^2+3(2x+y)^2+8y^2\,.
  \]Over $\Z_2[\sqrt{2}]$ the form $f$ is equivalent to the form $G$ in Lemma\;\ref{3.2}, because $\frac{16+2\sqrt{2}}{2\sqrt{2}}=1+4\sqrt{2}$ is a square in $\Z_2[\sqrt{2}]$ (\cite[63:1]{OMeara00}). This proves the inequality $g_{\Z[\sqrt{2}]}(2)\ge 5$.

  To prove the inequality in the other direction, consider a binary quadratic form $Q(x,\,y)$ over $\Z[\sqrt{2}]$ and suppose it can be written as a finite sum
  $Q=\sum_iL_i^2$ where $L_i=a_ix+b_iy$ with $a_i,\,b_i\in\Z_2[\sqrt{2}]$. Over a non-dyadic completion, every binary form is a sum of four linear forms (see e.g. \cite[Prop.\;3.3]{HeHuXu23}). Since the $\Z[\sqrt{2}]$-lattice $I_4$ has class number 1 (\cite[p.272, Satz 24]{Dzewas60}), $Q$ is a sum of 4 squares of linear forms over $\Z[\sqrt{2}]$ if and only if it is so over $\Z_2[\sqrt{2}]$. So by Lemma\;\ref{3.2}, we may assume $Q$ is equivalent to the form $G=2\sqrt{2}(x^2+xy+y^2)$ over $\Z_2[\sqrt{2}]$. This means that $Q$ represents only elements that are divisible by $2\sqrt{2}$.

  If all the $L_i^2$ represent only elements divisible by $2\sqrt{2}$, then $a^2_i$ and $b^2_i$ are divisible by $2\sqrt{2}$, hence $2$ divides $a_i$ and $b_i$. But this would imply that the integral values represented by $Q$ are all divisible by 4. This contradicts the fact that $2\sqrt{2}$ is represented by $Q$ over $\Z_2[\sqrt{2}]$.  Hence, one of $L_i^2$ represents an element not divisible by $2\sqrt{2}$. Then $Q-L_i^2$ is a sum of 4 linear forms over the dyadic completion by Lemma\;\ref{3.2}. Since $I_4$ has class number 1, this shows that $Q$ is a sum of 5 squares of linear forms over $\Z[\sqrt{2}]$.
\end{proof}

\noindent \emph{Acknowledgements.} We thank Prof.\,Fei Xu for helpful discussions. The authors were supported by a grant from the National Natural Science Foundation of China (no.\,12171223) and the Guangdong Basic and Applied Basic Research Foundation (no.\,2021A1515010396).

\addcontentsline{toc}{section}{\textbf{References}}

\bibliographystyle{alpha}

\bibliography{SOSQsqrt2}

%\clearpage \thispagestyle{empty}

%\printbibliography

Contact information of the authors:

\

Zilong HE

\medskip

School of Computer Science and Technology,

Dongguan University of Technology,

Dongguan 523808, China

Email: zilonghe@connect.hku.hk

\

Yong HU

\medskip

Department of Mathematics

Southern University of Science and Technology

%No. 1088, Xueyuan Blvd., Nanshan district

Shenzhen 518055, China

 %, Guangdong,China

Email: huy@sustech.edu.cn

\

\end{document}